\documentclass[10pt,leqno]{amsart}
\usepackage{euscript, amssymb, amsmath, amsthm}
\usepackage{epsfig}
\usepackage{graphicx}
\usepackage{longtable}
\usepackage{dcolumn}
\usepackage{mathrsfs}
\usepackage{color}
\usepackage{setspace}

\newcommand{\ud}{\mathrm{d}}

\title[Existence and regularity results for the inviscid PE in a channel]{\Small{Existence and regularity results for the inviscid primitive equations with lateral periodicity}}

\author[M. Hamouda, C. Jung and R. Temam]{Makram Hamouda$^{1,2}$, Chang-Yeol Jung$^{3}$ and Roger Temam$^{1}$}
\address{$^{1}$ The Institute for Scientific Computing and
Applied Mathematics, \\ Indiana University, 831 E. 3rd St., Rawles Hall,\\ Bloomington, IN 47405, USA}
\address{$^{2}$ University of Tunis El Manar, Facutly of Sciences of Tunis, Department of Mathematics,
Tunis, Tunisia.}
\address{$^3$ Department of Mathematical Sciences, School of Natural Science, Ulsan National Institute of Science and Technology, UNIST-gil 50, Ulsan 689-798, Republic of Korea}
\subjclass[2000]{} \keywords{}

 \newtheorem{thm}{Theorem}[section]
 \newtheorem{cor}{Corollary}[section]
 \newtheorem{lem}{Lemma}[section]
 \newtheorem{prop}[thm]{Proposition}

 \theoremstyle{remark}
 \newtheorem{rem}{Remark}[section]
  
 \numberwithin{equation}{section}

\newcommand{\en}{\Bbb N}
\newcommand{\er}{\Bbb R}
\begin{document}

\maketitle


\newcommand{\ep}{\varepsilon}

\begin{abstract}
The article is devoted to prove the existence and regularity of the solutions
of the $3D$ inviscid Linearized Primitive Equations (LPEs) in a channel with
lateral periodicity. This was assumed in a previous work \cite{HJT} which is concerned with the boundary layers generated by the corresponding viscous problem. Although the equations under investigation here are of hyperbolic type, the standard methods do not apply because of the specificity of the hyperbolic system. A set of \textit{non-local} boundary conditions for the inviscid LPEs has to be imposed at the top and bottom of the channel making thus the system well-posed.
\end{abstract}

\section{Introduction}
Following the same linearization process, as presented in \cite{rtt08}, we are interested in this article in the existence and uniqueness of the solutions of the inviscid Linearized Primitive Equations (LPEs) of the ocean that we state as below:
\begin{align}\label{model0}
\left\{%
\begin{array}{l}
u_t + \bar{U}_0u_x - f v + \phi_x = F_u,\\
v_t + \bar{U}_0 v_x + f u + \phi_y   =
F_v,\\
\psi_t + \bar{U}_0\psi_x + N^2w  = F_{\psi},\\
\phi_z = \psi,\\
u_x + v_y + w_z = 0.
\end{array}%
\right.
\end{align}
Here the domain $\mathcal{M}$ is a cube of $\er^3$, {\it i.e.} $\mathcal{M} =
\mathcal{M}'\times(-L_3,0)$ with $\mathcal{M}'=(0,L_1)\times(0,L_2)$. Note that the external force ${\bf F}=(F_u, F_v, F_{\psi})^T$ may not have a significant physical meaning; it is introduced here for mathematical generalisation and to possibly handle non homogeneous boundary conditions which we do not consider in this article. The unknowns $(u, v, w), \psi$ and $\phi$ denote, respectively, the velocity of the fluid, the temperature and the pressure. The constant  $\bar{U}_0 > 0$ is the first component of the uniform stratified velocity flow $(\bar{U}_0,0,0)$ around which the (full) nonlinear PEs are linearized; see \cite{rtt08} for more details. The function  $f = f(y)$ is the Coriolis parameter. Throughout this article, we will assume that $f$ is independent of $y$ and thus constant.\\

Equations (\ref{model0}) will be supplemented with lateral boundary conditions that we will discuss at length. Furthermore, as in \cite{rtt08}, the following boundary conditions are imposed at top and bottom:
\begin{equation}\label{z-bc}
(\frac{\partial u}{\partial z},\frac{\partial v}{\partial
z},\psi,w) = 0\ \text{at} \ z=0, -L_3;
\end{equation}
see \cite{rtt08} for more details.

It is now well-known that no set of {\it local} boundary conditions can guarantee the well-posedness of (\ref{model0}), see \cite{os} and \cite{TT}. However, many choices of  {\it nonlocal} boundary conditions are possible. In view of the hyperbolic character of the system (\ref{model0}), the authors considered in \cite{rtt08} a set of nonlocal boundary conditions inherited from a normal modal decomposition and then proved the existence and uniqueness of the solution of (\ref{model0}). Here, as we will see later, the boundary conditions are different from those in \cite{rtt08}; this justifies the study of the existence for the system (\ref{model0}). Indeed, we will adopt the boundary conditions suggested by the boundary layer analysis for the viscous PEs on a cube with lateral periodicity as in \cite{HJT}; see also \cite{HJT09}. Moreover, the article complements the study carried out in \cite{HJT} in which the existence and regularity of the inviscid solution (the limit of the viscous solutions) was assumed.

The issue of lateral boundary conditions for the primitive equations and related equations has been addressed in a number of articles; see e.g. besides the already quoted articles, \cite{363, chghagre, GHT1, ht07, HT08, JT05, pete2, Raymond, sltt}. We recall that the theory of the primitive equations has been initiated in \cite{ltw192, ltw292} and further developed by many authors, in particular in \cite{ct07,ko}. Many other singular perturbation problems and related issues are considered in e.g. \cite{chghagre, casa, tewa02, tewa2}.

The article is organized as follows. First, in Section  \ref{sec2}, we decompose the solutions of (\ref{model0}) in a modal basis with respect to the vertical direction $Oz$. Two kind of modes are then considered, the zero mode and the nonzero modes. Then, Section  \ref{sec3} is devoted to the study of the existence for the nonzero modes. However, the existence result of the zero mode solution is detailed in Section \ref{sec-mode0}. The results obtained in Section \ref{sec-mode0} are mainly inspired from \cite{CST09} with the necessary changes due to the periodicity condition considered in this article.  Therefore, based on the existence results of each mode, we prove in Section \ref{sec4} the existence of the (global) solution of (\ref{model0}). Finally, in Section \ref{sec6}, we state and prove the regularity results which were assumed in the boundary layer analysis study carried out in \cite{HJT}.

\section{The modal decomposition and the boundary conditions}\label{sec2}
In this section we start by giving an equivalent modal system to (\ref{model0}). Then, we aim to study the existence and uniqueness of the solution to each of the modal equations when they are associated, as we will see later on, with the appropriate boundary conditions. \\
We consider as in \cite{TT} the following modal decomposition:
\begin{subequations}\label{modal_form}
\begin{align}
(u,v,\phi) &= \sum_{n \geq
0}\mathcal{U}_n(z)(u_n,v_n,\phi_n)(x,y,t),\\
(w,\psi) &= \sum_{n \geq 1}\mathcal{W}_n(z)(w_n,\psi_n)(x,y,t);\
\text{note $w_0=\psi_0=0$},
\end{align}\label{modal_form_b}
\end{subequations}
where $\mathcal{U}_0=\frac{1}{\sqrt{L_3}}$,
$\mathcal{U}_n=\sqrt{\frac{2}{L_3}}\cos(\lambda_n z)$ and
$\mathcal{W}_n=\sqrt{\frac{2}{L_3}}\sin(\lambda_n z)$\footnote{These functions are unique up to some multiplicative constants and are deduced from the boundary conditions in the $z$ direction, namely (\ref{z-bc}).}, and the
frequencies $\lambda_n$ are given by:
\begin{align}
\lambda_n = \frac{n\pi}{L_3}, ~~~n \in \en.
\end{align}
The modal equations for $(u_n,v_n,\psi_n,\phi_n,w_n)$ are then given in $\mathcal{M}'$ by
\begin{align}\label{mode_zero}
\left\{%
\begin{array}{l}
u_{0t} + \bar{U}_0u_{0x} - f v_0 + \phi_{0x}  = F_{u_0},\\
v_{0t} + \bar{U}_0 v_{0x} + f u_0 + \phi_{0y}  =F_{v_0},\\
u_{0x}+v_{0y} = 0,
\end{array}%
\right.
\end{align}
for $n=0$, and,
\begin{align}\label{moden}
\left\{%
\begin{array}{l}
u_{nt} + \bar{U}_0u_{nx} - f v_n + \phi_{nx} = F_{u_n},\\
v_{nt} + \bar{U}_0 v_{nx} + f u_n + \phi_{ny} =F_{v_n},\\
\psi_{nt} + \bar{U}_0\psi_{nx} + N^2w_n = F_{\psi_n},\\
\phi_n = -\lambda_n^{-1}\psi_n,\\
w_n = -\lambda_n^{-1}(u_{nx}+v_{ny}),
\end{array}%
\right.
\end{align}
for $n \geq 1$.
\begin{rem}\label{rem21}
Thanks to (\ref{moden})$_{4,5}$, the solutions $\phi_n$ and $w_n$ are deduced respectively from $\psi_n$ and $u_n, v_n$. Hence, we will omit these quantities hereafter in our study and we will denote by $U_n:=(u_n,v_n,\psi_n)^T, n \ge 1,$ (resp. $U_0:=(u_0,v_0)^T)$ the solution of (\ref{moden}) (resp. (\ref{mode_zero})) when this solution exists. Hereafter, we denote by $\mathcal{A}_n$, for $n \ge 1$, the differential operators acting on $U_n:=(u_n,v_n,\psi_n)$,
\begin{align}\label{operator-A}
\mathcal{A}_n U_n=\left\{%
\begin{array}{l}
\bar{U}_0u_{nx} - f v_n -\frac{1}{\lambda_n}\psi_{nx},\\
\bar{U}_0 v_{nx} + f u_n -\frac{1}{\lambda_n}\psi_{ny},\\
\bar{U}_0\psi_{nx} - \frac{N^2}{\lambda_n} (u_{nx}+v_{ny}).
\end{array}%
\right.
\end{align}
The complete definition of the operator $\mathcal{A}_n$ will be introduced later on in this article (see (\ref{operator-A}) and (\ref{domain-A}) below).
\end{rem}
Now, we can state the different boundary conditions that we will associate with the equations (\ref{mode_zero}) and (\ref{moden}). As mentioned before, these boundary conditions are borrowed from \cite{HJT}.
We will hereafter consider two types of modes; the zero mode and the (nonzero) $n$th mode for $n\ge 1$. Note that the nonzero modes are all supercritical (see \cite{HJT09} for more details about this notion).

First of all we give the initial data which are valid for all modes:
\begin{equation}\label{ini_cond_mode}
(u_n,v_n,\psi_n) =
(\tilde{u}_{n},\tilde{v}_{n},\tilde{\psi}_{n})(x,y)\ \text{at
$t=0$}, \quad \forall \ n \in \en;
\end{equation}
here we note that $\psi_0 \equiv 0$ since $\mathcal{W}_0=0$.

In a first step, we will focus on the existence of solution for the n{\it th} modes $n \ge 1$. For these modes, the boundary conditions at $y=0,L_2$ read as follows (see \cite{HJT}):
\begin{align}
&\zeta_n(x,L_2,t)=(v_n+N^{-1}\psi_n)(x,L_2,t) = 0,\label{zeta}\\
&\chi_n(x,0,t)=(v_n-N^{-1}\psi_n)(x,0,t) = 0,\label{xhi}
\end{align}
where
\begin{align}\label{combi1}
\zeta_n = v_n + N^{-1}\psi_n,\quad \text{and} \quad \chi_n = v_n - N^{-1}\psi_n, \quad \forall \ n\ge 1.
\end{align}
Moreover, we supplement these conditions with a periodicity condition in the $x$ direction:
\begin{equation}\label{per_bdry_x}
(u_n, v_n, \psi_n) (0,y,t) = (u_n, v_n, \psi_n) (L_1,y,t).
\end{equation}
In a second step, we will state and give the proof of the existence of solution for the mode zero which is also called the Barotropic mode of the primitive equations. In view of the particularity of the the zero mode equations, the boundary conditions of this mode are inspired from the usual boundary conditions that we associate with the Euler type equations with rotating term, see e.g. \cite{ht07}, \cite{CST09} or \cite{tewa2}. Namely, we impose that
\begin{align}
&u_0(0,y,t)=u_0(L_1,y,t), \ \forall \ y \in (0,L_2), \ \forall \ t \in (0,T),\label{mode_zero_limit_bdry}\\
&v_0(0,y,t)=v_0(L_1,y,t),\ \forall \ y \in (0,L_2),  \ \forall \ t \in (0,T),\label{mode_zero_limit_bdry2}\\
&v_0 = 0\ \text{at $y=0,L_2$,}\ \forall \ x \in (0,L_1),  \ \forall \ t \in (0,T).\label{mode_zero_limit_bdry1}
\end{align}
The zero mode is in fact modeling a rotating fluid and though its corresponding equations look simple, they are not and the study of the existence will be deduced from \cite{CST09}, see also \cite{ht07}. The approach used for treating the existence for the mode zero in \cite{CST09} is different since the situation in the present article is slightly different because of the periodicity in $x$.\\

We now have all the necessary initial and boundary conditions to prove the existence of the $n${\it th} mode solution for any $n \in \en$. This will be done in the next sections.
\section{Existence of the $n${\it th} mode solution, $n \ge 1$}\label{sec3}
In this section we will rewrite the limit problem of the $n${\it th} mode for which we aim to prove the existence using the linear semi-group theory and the Hille-Philipps-Yoshida Theorem. For that purpose, we need to introduce the function spaces in which the solution is shown to exist. Moreover, some regularity results are necessary in order to give a sense to the boundary values cited above and to argue the regularity results assumed in the asymptotic analysis carried out in  \cite{HJT}.

{\it In what follows in this section, all the unknown functions and operators depend on $n$}. However, the subscript $n$ will be omitted, and it will be reintroduced again when necessary. In particular,  we denote by $(u_n, v_n, \psi_n)=(u, v, \psi)$ and so on for the other quantities when we need to recall them.

Using now the notations and remarks cited before, we write the limit problem (\ref{moden}) as follows:
\begin{align}\label{mode-n}
\left\{%
\begin{array}{l}
u_{t} + \bar{U}_0u_{x} - f v -\frac{1}{\lambda}\psi_{x} = F_{u},\\
v_{t} + \bar{U}_0 v_{x} + f u -\frac{1}{\lambda}\psi_{y} =F_{v},\\
\psi_{t} + \bar{U}_0\psi_{x} - \frac{N^2}{\lambda} (u_{x}+v_{y}) = F_{\psi},
\end{array}%
\right.
\end{align}
which we supplement with the following boundary conditions:
\begin{align}\label{bc}
\left\{%
\begin{array}{l}
\zeta(x,L_2,t):=(v+N^{-1}\psi)(x,L_2,t) = 0,\\
\chi(x,0,t):=(v-N^{-1}\psi)(x,0,t) = 0,\\
(u, v, \psi) (0,y,t) = (u, v, \psi) (L_1,y,t).
\end{array}%
\right.
\end{align}

Then, we rewrite (\ref{mode-n})-(\ref{bc}) in a suitable Hilbert space $H \subset {\bf
L^2(\mathcal{M'})}:=L^2(\mathcal{M'})^3$, in the following abstract form:
\begin{align}\label{abstract-pb}
\left\{%
\begin{array}{l}
\displaystyle{\frac{\ud U}{\ud t}+A U= F,\quad \forall \,t > 0,}\vspace{.2cm}\\
U(0) = \widetilde{U},
\end{array}%
\right.
\end{align}
where $U : \er_+ \rightarrow H$, and $A$ is a linear unbounded operator in $H$ with domain $D(A) \subset H$ which will be specified later on in this section (see (\ref{domain-A})). The space $H$ is simply given by:
\begin{equation}\label{spaceH}
H={\bf L^2_{per}(\mathcal{M'})}:=\left\{ U \in {\bf
L^2(\mathcal{M'})} ~\textnormal{s.t.}~ U \ \text{is ~$L_1-$periodic in} \ x\right\}.
\end{equation}

In the following, we define the spaces $\mathcal{X}$ and $D(A)$. For that purpose, we first endow the space ${\bf
L^2(\mathcal{M'})}$ with the scalar product and norm:
\begin{align}\label{inner_all}
(U,U^{*}) = \int_{\mathcal{M'}} (u u^{*} + v v^{*} +
\frac{1}{N^2}\psi \psi^{*})d\mathcal{M'},\quad |U|_{\bf L^2} =
(U,U)^\frac{1}{2}.
\end{align}
We then introduce the auxiliary space for which we prove a trace theorem that we will use later on to define $D(A)$:
\begin{equation}
\mathcal{X}(\mathcal{M'}):=\left\{ U \in H ~\textnormal{s.t.}~ \mathcal{\bar{A}} U \in {\bf
L^2(\mathcal{M'})}\right\},
\end{equation}
endowed with its natural Hilbert norm $(|U|^2_{\bf L^2(\mathcal{M'})} + |\mathcal{\bar{A}} U|^2_{\bf L^2(\mathcal{M'})})^{1/2}$.\\
Here, $\mathcal{\bar{A}}=\mathcal{A}+\mathcal{B}$ denotes the differential operator $\mathcal{\bar{A}}=(\mathcal{\bar{A}}_1, \mathcal{\bar{A}}_2, \mathcal{\bar{A}}_3)$ operating on $U=(u,v,\psi)$ as follows:
\begin{align}\label{operator-A}
\mathcal{A} U =\left\{%
\begin{array}{l}
\bar{U}_0u_{x} -\frac{1}{\lambda}\psi_{x},\\
\bar{U}_0 v_{x}  -\frac{1}{\lambda}\psi_{y},\\
\bar{U}_0\psi_{x} - \frac{N^2}{\lambda} (u_{x}+v_{y}),
\end{array}%
\right.
\quad \text{and} \quad
\mathcal{B} U =\left\{%
\begin{array}{l}
- f v,\\
f u,\\
0.
\end{array}%
\right.
\end{align}
Note that $\mathcal{B}$ is a linear continuous operator on ${\bf
L^2(\mathcal{M'})}$. Hence, it is sufficient to prove the Hille-Phillips-Yosida Theorem for $\mathcal{A}$ which involves, among other results, to show the positivity of $\mathcal{A}$ and its adjoint.
However, some auxiliary steps are necessary in order to make the computations well-defined. First, we will introduce the space $D(A)$, the domain of $A$, as follows:
\begin{equation}\label{domain-A}
D(A)=\left\{U=(u,v,\psi) \in  {\bf L^2(\mathcal{M'})} ~\textnormal{s.t.}~ \mathcal{A} U \in  {\bf L^2(\mathcal{M'})} ~\textnormal{and}~ U \ \text{satisfies} \ (\ref{bc})_{1,2} \right\}.
\end{equation}

\begin{rem}
We recall here that all the function spaces defined above; $H, \mathcal{X}(\mathcal{M'})$ and $D(A)$, are in fact dependent of $n$ ($n \ge 1$) and should be denoted respectively $H^n, \mathcal{X}^n(\mathcal{M'})$ and $D(A^n)$ when it is necessary to reintroduce the superscript $n$, especially in Section \ref{sec4}. The function spaces for $n=0$ are borrowed from \cite{CST09} and will be redefined in Section \ref{sec-mode0}.
\end{rem}

We now state and prove several trace results for $U=(u,v,\psi)$ giving thus a sense to the definition of $D(A)$. We have the following lemma.
\begin{lem}\label{lem-sapce xi}
For all $U=(u,v,\psi) \in \mathcal{X}(\mathcal{M'})$, the traces of $\psi$ and $v$ are well-defined at $y=0,L_2$. Moreover, the function $U=(u,v,\psi)$ is periodic in $x$ and its trace is well-defined at $x=0, L_1$.
\end{lem}
\begin{proof}
For $U=(u,v,\psi) \in \mathcal{X}(\mathcal{M'}) \subset L^2_x(0,L_1; (L^2_y(0,L_2))^3)$, we have $\mathcal{A} U \in {\bf
L^2(\mathcal{M'})}$. In particular, we infer that $U_x$ belongs to $L^2_x(0,L_1; (H^{-1}_y(0,L_2))^3$. Now, since $(\mathcal{A} U)_2 \in L^2(\mathcal{M'})$, we conclude that $\psi_y \in L^2_y(0,L_2; H^{-1}_x(0,L_1)$ and consequently $\psi \in \mathcal{C}_y([0,L_2]; H^{-1}_x[0,L_2]$. Hence, the trace of $\psi$ is well-defined at $y=0,L_2$. Similarly, from $(\mathcal{A} U)_3 \in L^2(\mathcal{M'})$, we deduce that $v \in \mathcal{C}_y([0,L_2]; H^{-1}_x(0,L_1))$ and its trace is also well-defined at $y=0,L_2$.\\
Now, combining $(\mathcal{A} U)_1 ~\textnormal{and}~ (\mathcal{A} U)_3 \in L^2(\mathcal{M'})$, we obtain that $\psi_x$ and $u_x$ belong to $L^2_x(0,L_1; H^{-1}_y(0,L_2)$ and then $\psi, u \in \mathcal{C}_x([0,L_1]; H^{-1}_y(0,L_2))$. Thus, the periodic boundary conditions in $x$ for $u$ and $\psi$ make sense since $U \in H$.\\
Finally, using   $(\mathcal{A} U)_3 \in L^2(\mathcal{M'})$ and the above results in this proof, we conclude that $v_x \in L^2_x(0,L_1; H^{-1}_y(0,L_2))$ and then $v \in \mathcal{C}_x([0,L_1]; H^{-1}_y(0,L_2))$ which guarantees the definition of the periodic boundary conditions for $v$.\\
This concludes the proof of the lemma.
\end{proof}

Now, thanks to Lemma \ref{lem-sapce xi}, the space $D(A)$ is well-defined. Its adjoint $A^*$ is classically defined as follows \cite{rudin}: its domain $$D(A^*)\!=\!\Big\{\!U^\sharp \in H,\ V \rightarrow (U^\sharp,A V)\ \text{is continuous on $D(A)$ for the norm of $H$}\!\Big\}.$$ It can be shown as in e.g. \cite{rtt08}, that $$D(A^{*})\!=\!\left\{\!U^\sharp=(u^\sharp,v^\sharp,\psi^\sharp) \!\in\!  {\bf L^2(\mathcal{M'})} \,\textnormal{s.t.}\, \mathcal{A^{*}} U^\sharp \!\in\!  {\bf L^2(\mathcal{M'})} \,\textnormal{and}\, U^\sharp \ \text{satisfies} \ (\ref{bc-a*})\!\right\},$$
and for $U^\sharp \in D(A^*)$, $A^*U^\sharp = \mathcal{A}^*U^\sharp$. Here
\begin{align}\label{operator-A*}
\mathcal{A^{*}} U^\sharp =\left\{%
\begin{array}{l}
-\bar{U}_0 u^\sharp_{x} + \frac{1}{\lambda} \psi^\sharp_{x},\\
-\bar{U}_0 v^\sharp_{x} + \frac{1}{\lambda} \psi^\sharp_{y},\\
-\bar{U}_0 \psi^\sharp_{x} + \frac{N^2}{\lambda} (u^\sharp_{x}+v^\sharp_{y}).
\end{array}%
\right.
\end{align}
We associate with the operator $\mathcal{A^{*}}$ the following boundary conditions:
\begin{align}\label{bc-a*}
\left\{%
\begin{array}{l}
(v^\sharp-N^{-1} \psi^\sharp)(x,L_2,t) = 0,\\
(v^\sharp+N^{-1} \psi^\sharp)(x,0,t) = 0,\\
(u^\sharp,v^\sharp,\psi^\sharp) (0,y,t) = (u^\sharp,v^\sharp,\psi^\sharp) (L_1,y,t).
\end{array}%
\right.
\end{align}

\subsection{The positivity of $\mathcal{A}$ and $\mathcal{A^{*}}$}

In this paragraph, we will prove the positivity of the operators $\mathcal{A}$ and $\mathcal{A^{*}}$. This will allow us to apply the Hille-Phillips-Yosida Theorem and conclude the existence of the mode solutions for $n \ge 1$. Since we do not have yet proved the necessary regularity results for the mode solutions, we proceed by regularization and we start by showing the positivity of $\mathcal{A}$.
\begin{lem}\label{pos-A}
The operator $A$ is positive on $D(A)$, $(AU,U) \geq 0 \ \forall \, U \in D(A)$.
\end{lem}
\begin{proof}
Given $U \in D(A)$, we approximate $U=(u,v,\psi)$ by  smooth functions $U_{\varepsilon}=(u_{\varepsilon},v_{\varepsilon},\psi_{\varepsilon})$ using a partial regularisation in the $x$ variable. More precisely, we first extend the functions $u,v,\psi$ to $\er_x$ (that is $\er_x \times (0, L_2)$), by periodicity in $x$, and we denote the extended functions by $\tilde{u},\tilde{v},\tilde{\psi}$. Then, for $\ep >0$, we introduce a regularizing function $\rho_{\ep}=\ep^{-1} \rho (x/\ep)$ where $\rho$ is such that
\begin{displaymath}
\rho \ge 0, \
\rho \in \mathcal{C}^{\infty}_c(\er) \ \text{and} \
\int_{\er}\rho (x) dx=1.
\end{displaymath}
We write $*_x$ the partial convolution with respect to $x$ and we call $u_{\ep},v_{\ep},\psi_{\ep}$ the restrictions to $\Omega$ of $\rho_{\ep}*_x\tilde{u},\rho_{\ep}*_x\tilde{v},\rho_{\ep}*_x\tilde{\psi}$. The functions $u_{\ep},v_{\ep},\psi_{\ep}$ belong to $L^2_y((0,L_2); \mathcal{C}^{\infty}([0,L_1]))$ and, as $\ep \to 0$, $u_{\ep}\to u,v_{\ep}\to v,\psi_{\ep}\to \psi$ in $L^2(\Omega)$.

Furthermore, if $u,v,\psi$ satisfy the system $\mathcal{A} U= F=(F_u, F_v, F_\psi)$, then by extension by periodicity, convolution with $\rho_{\ep}$ and restriction to $\Omega$, we find that:
\begin{equation}\label{system-conv}
\left\{%
\begin{array}{l}
\bar{U}_0u_{\ep x} -\frac{1}{\lambda}\psi_{\ep x}= F_{u_{\ep}}=\rho_{\ep}*_x F_{u},\\
\bar{U}_0 v_{\ep x}  -\frac{1}{\lambda}\psi_{\ep y}= F_{v_{\ep \ep}}=\rho_{\ep}*_x F_{v},\\
\bar{U}_0\psi_{\ep x} - \frac{N^2}{\lambda} (u_{\ep x}+v_{\ep y})= F_{\psi_{\ep}}=\rho_{\ep}*_x F_{\psi}.
\end{array}%
\right.
\end{equation}
Moreover, the functions  $u_{\ep},v_{\ep},\psi_{\ep}$ satisfy the boundary conditions (\ref{zeta}), (\ref{xhi}) and (\ref{per_bdry_x}), so that $U_{\ep} \in D(A)$ and
\begin{equation}\label{system-reg}
\mathcal{A} U_{\ep} = F_{\ep}=\rho_{\ep}*_x F.
\end{equation}
Now, using the equations (\ref{system-conv})$_{1,2}$, we see that $v_{\ep y}, \psi_{\ep y} \in L^2_y((0,L_2); \mathcal{C}_{per}^{\infty}([0,L_1]))$, and hence
\begin{equation}
v_{\ep}, \psi_{\ep} \in H^1_y((0,L_2); \mathcal{C}_{per}^{\infty}([0,L_1])).
\end{equation}
Finally, the functions $u_{\ep},v_{\ep},\psi_{\ep}$ are now regular and we can easily prove the positivity of the operator $A$ on the space $D(A)$. For that, it suffices to remember the definition of the scalar product (\ref{inner_all}) and use the boundary conditions (\ref{bc}); this yields for $U_{\ep} \in D(A)$:
\begin{eqnarray}\label{posit-regular}
(A U_{\ep}, U_{\ep}) &=& \frac{1}{\lambda} \int_0^{L_1}[v_{\ep} \psi_{\ep} (y=0)-v_{\ep} \psi_{\ep} (y=L_2)] \ud x \\
&=& \frac{N}{4\lambda} \int_0^{L_1}[\zeta^2_{\ep} (y=0)+\chi^2_{\ep} (y=L_2)] \ud x \ge 0,\nonumber
\end{eqnarray}
where $\chi_{\ep}, \zeta_{\ep}$ are defined as $\chi, \zeta$ in (\ref{zeta})-(\ref{xhi}) with $v, \psi$ replaced by $v_{\ep}, \psi_{\ep}$, respectively.\\
At the limit $\ep \to 0$, $U_{\ep}\to U$ in $\mathbf{L}^2(\Omega)$ as we already said and, similarly, by (\ref{system-reg}), $\mathcal{A} U_{\ep}\to \mathcal{A} U$ in $\mathbf{L}^2(\Omega)$. So (\ref{posit-regular}) gives the positivity of $\mathcal{A}$ by passing to the limit as $\ep \to 0$.\\
This achieves the proof of the lemma.
\end{proof}
We can now define the traces of the functions in $D(A^{*})$ similarly as we did for $D(A)$, and we proceed with the positivity of the adjoint operator $\mathcal{A^{*}}$.
\begin{lem}\label{pos-A*}
The operator $\mathcal{A^{*}}$, defined by (\ref{operator-A*}),  is positive on $D(A^{*})$.
\end{lem}
\begin{proof}
Let $U^\sharp=(u^\sharp,v^\sharp,\psi^\sharp) \in {\bf L^2(\mathcal{M'})}$, with $\mathcal{A^{*}} U^\sharp \in {\bf L^2(\mathcal{M'})}$. Then, we may assume that the $U^\sharp$ is a smooth function. Indeed, it suffices for that to proceed by regularization as we did for the positivity of $\mathcal{A}$ in the proof of Lemma \ref{pos-A}, that is we can assume $u^\sharp,v^\sharp,\psi^\sharp$ and $v^\sharp_y,\psi^\sharp_y$ belong to $L^2_y((0,L_2); \mathcal{C}^{\infty}([0,L_1])$.
Hence, using the boundary conditions (\ref{bc-a*}), we obtain
\begin{eqnarray}\label{posit-regular*}
(\mathcal{A^{*}} U^\sharp, U^\sharp) &=& \frac{1}{\lambda} \int_0^{L_1}[v^\sharp \psi^\sharp (y=0)-v^\sharp \psi^\sharp (y=L_2)] \ud x \\
&=& \frac{N^2}{4\lambda} \int_0^{L_1}[v^{\sharp2} (y=0)+v^{\sharp2} (y=L_2)] \ud x \ge 0.\nonumber
\end{eqnarray}
This proves the positivity of $\mathcal{A^{*}}$ and ends thus the proof of the lemma.
\end{proof}

We are now able to state our existence result of which the proof is already done thanks to the positivity results obtained in Lemma \ref{pos-A} and Lemma \ref{pos-A*}. Indeed with the closedness of $A$ and $A^*$ which is easy to prove (and $D(A), D(A^*)$ dense in $H$), these lemmas imply that $-A$ is the infinitesimal generator of a semigroup of contractions in $H=L^2(\Omega)^3$. More precisely, we have the following existence result.

\begin{thm}\label{th3.1}
For $F$ given in $\mathcal{C}^1(0,T;H)$ and $U_0$ given in $D(A)$, there exists a unique solution $U$ to the system (\ref{abstract-pb}) with
\begin{align}\label{exist-sol}
\left\{%
\begin{array}{c}
U \in \mathcal{C}^1([0,T];H) \cap \mathcal{C}^0([0,T];D(A)),\\
\displaystyle{\frac{\ud U}{\ud t}\in \mathcal{C}^0([0,T];H), \ \forall \ T>0}.
\end{array}%
\right.
\end{align}
\end{thm}

\begin{rem}\label{remark4.1}
Remember that Theorem \ref{th3.1} relates to the modes $n, n\ge 1$. Then, as in \cite{rtt08}, we have the additional task of combining all the modes $n\ge 1$ together with the mode $n=0$. This will be done in Section \ref{sec6}. Furthermore, concerning the mode $n$, it is sufficient to observe that by differentiating (\ref{abstract-pb})$_1$ (with the dependence in $n$ reintroduced here) once in time, multiplying the obtained equation by $U_{n}'$ and using the Gronwall inequality, Theorem \ref{th3.1} gives in fact that
\begin{equation}
|U_{n}'(t)|^2_{H} \le \kappa |U_{n}'(0)|^2_{H} + \int_0^t |F_{n}'(s)|^2_{H} \ud s,\footnote{Here the space $H$ which is given by (\ref{spaceH}); in fact $H=H_n$.}
\end{equation}
where $\kappa$ is a constant which may change of value from one place to another while remaining independent of $n$. Using $U_{n}'(0)=F_{n}(0)-A_n U_n(0)=F_{n}(0)-A_n \widetilde{U}_n,$ (\ref{operator-A}) and  the fact that $\lambda_n \to \infty$ as $n \to \infty$, we infer that
\begin{equation}\label{bound1}
|U_{n}'(t)|^2_{H} \le \kappa \left(|F_{n}(0)|^2_{H}+|\widetilde{U}_n|^2_{H^1(\mathcal{M'})^3}\right) + \int_0^t |F_{n}'(s)|^2_{H} \ud s.
\end{equation}
Moreover, by direct energy estimates of (\ref{abstract-pb}) we easily obtain the following norm continuity estimate, with $\kappa$ again independent of $n$:
\begin{equation}\label{bound2}
|U_{n}(t)|_{H} \le \kappa |\widetilde{U}_n|_{H}, \quad \forall \ t \in (0,T).
\end{equation}
When we deal with the existence of the concatenated solution in Section \ref{sec4} below, the estimates (\ref{bound1})-(\ref{bound2}) ensure the boundedness properties of the   concatenated solution as it is needed for the Hille-Phillips-Yosida Theorem.
\end{rem}

\section{Existence and regularity results for the zeroth mode}\label{sec-mode0}
We deal in this section with the existence of solution for the zeroth mode solution of (\ref{mode_zero}) together with the boundary and initial conditions (\ref{mode_zero_limit_bdry})--(\ref{mode_zero_limit_bdry1}) and (\ref{ini_cond_mode}) (for $n=0$), respectively. For that purpose, we will follow the same steps as developed in \cite{CST09}. However, we have to adapt here the computations which are somehow simpler thanks to the periodicity in $x$. Indeed, the first space derivative term of the velocity in (\ref{mode_zero})$_{1,2}$ is taken with respect to the periodic variable $x$. This is a simplification in studying the existence for the zero mode. Hence, the proof follows in part that of \cite{CST09} with the changes introduced above. See below for the details.\\

In considering the zeroth mode, we temporarily drop the Coriolis force which corresponds to a linear bounded operator. Now, let us introduce the function spaces $H^0, \mathcal{X}^0(\mathcal{M'})$ and $D(A^0)$ corresponding to the stationary operator associated with the zero mode system (\ref{mode_zero})$_{1,2}$, namely
\begin{equation}\label{space-glob-2}
H^0 =\{{\bf u}=(u,v) \in L^2(\mathcal{M'})^2 \ \text{s.t.} \ u_x+v_y=0, \ \text{and} \ {\bf u}
\ \text{satisfies} \ (\ref{mode_zero_limit_bdry}) \ \text{and} \ (\ref{mode_zero_limit_bdry1}) \},
\end{equation}
and
\begin{equation}\label{space-glob-3}
\mathcal{X}^0(\mathcal{M'}) = \left\{ {\bf u} \in H^0, \exists \ \varphi \in \mathcal{D}'(\mathcal{M'}) \ \text{s.t.} \ \mathcal{A}^0{\bf u} \in H^0 \right\},
\end{equation}
with
\begin{align}\label{operator-A0}
\mathcal{A}^0{\bf u} =\left\{%
\begin{array}{l}
-\bar{U}_0 u_{x} + \varphi_{x},\\
-\bar{U}_0 v_{x} + \varphi_{y}.
\end{array}%
\right.
\end{align}
Note that the pressure $\varphi$ is unique up to an additive constant.

Now, we can define the traces of the zero mode solution as it is stated in the following lemma.
\begin{lem}
For ${\bf u}=(u,v) \in \mathcal{X}^0(\mathcal{M'})$, the traces of $v$ and $\varphi$ are well-defined on all sides of $\partial\mathcal{M'}$ and belong to $H_x^{-1}(0,L_1), H_y^{-1}(0,L_2)$. Moreover, the traces of $u$ at $x=0, L_1$ are defined  and belong to $H_y^{-1}(0,L_2)$.
\end{lem}
\begin{proof}
The proof is similar to that of Theorem 2.1 in \cite{CST09}.
\end{proof}

Finally, we can introduce the space $D(A^0)$ which is defined as follows:
\begin{equation}\label{domain-A0}
D(A^0) = \left\{
{\bf u}=(u,v) \in \mathcal{X}^0(\mathcal{M'}), {\bf u}\ \text{satisfies} \ (\ref{mode_zero_limit_bdry2})\right\}.
\end{equation}
Note that $A^0$ is closed (straightforward with the use of sequences), and the space $D(A^0)$ is dense in $H^0$ since $D(A^0) \supset \mathcal{V}$ where
\begin{displaymath}
\begin{array}{c}
\mathcal{V}=\{{\bf u}=(u,v) \in \mathcal{C}^{\infty}(\mathcal{M'}), \ \text{s.t.} \ {\bf u} \ \text{has compact support in $(0,L_2)$}\\ \text{and is periodic in} \ x\ \ \text{and} \ u_x+v_y=0\};
\end{array}
\end{displaymath}
we set $A^0 {\bf u}=\mathcal{A}^0{\bf u},$ for ${\bf u} \in D(A^0)$.

In the following, we will define the adjoint operator $A^{0*}$ of $A^0$. Similarly as for $A^0$, we introduce the function spaces related to the adjoint operator $A^{0*}$:
\begin{equation}\label{space-glob-3*}
\mathcal{X}^{0*}(\mathcal{M'}) = \left\{ {\bf u} \in H^0, \exists \ \varphi \in \mathcal{D}'(\mathcal{M'}) \ \text{s.t.} \ \mathcal{A}^*_0{\bf u} \in H^0 \right\},
\end{equation}
with
\begin{align}\label{operator-A0*}
\mathcal{A}^*_0{\bf u} =\left\{
\begin{array}{l}
-\bar{U}_0 u_{x} +fv+ \varphi_{x},\\
-\bar{U}_0 v_{x} -fu+ \varphi_{y}.
\end{array}%
\right.
\end{align}
Then, the domain of $A^{0*}$, \cite{rudin}, is simply given by
\begin{equation}\label{domain-A0*}
D(A^{0*}) = \left\{
{\bf u}=(u,v) \in \mathcal{X}^{0*}(\mathcal{M'}), {\bf u}\ \text{satisfies} \ (\ref{mode_zero_limit_bdry2})\right\},
\end{equation}
and we set $A^{0*} {\bf u}=\mathcal{A}_0^*{\bf u}, \forall {\bf u} \in D(A^{0*})$.\\
It is also worth noting that the space $D(A^{0*})$ which contains $\mathcal{V}$ is dense in $H^0$ and the operator $A^{0*}$ is closed.
\begin{rem}\label{rem-reg-mode0}
We note that for ${\bf u} \in D(A^{0})$ (the same assertion holds for $D(A^{0*})$), we have ${\bf u}_x \in L^2(\mathcal{M'})^2$ and the corresponding $\varphi$ vanishes. Indeed, using the definition of $D(A^{0})$ as in (\ref{domain-A0}) and more precisely the definition of the auxiliary space $\mathcal{X}^0(\mathcal{M'})$ as in (\ref{space-glob-3}), we have the existence of $\varphi \in \mathcal{D}'(\mathcal{M'})$ such that $\mathcal{A}^{0} {\bf u}={\bf f}=(f_1, f_2) \in H^0$. Hence, writing that $\textnormal{div}\, \mathcal{A}^{0} {\bf u}=0$, we find that $\varphi$ is a harmonic function. Writing in addition the boundary conditions from (\ref{space-glob-2}) we find that
\begin{displaymath}
(-\bar{U}_0 u_{x} + \varphi_{x})|_{x=0}=(-\bar{U}_0 u_{x} + \varphi_{x})|_{x=L_1},
\end{displaymath}
and since $u|_{x=0}=u|_{x=L_1}$ this implies that $\varphi|_{x=0}=\varphi|_{x=L_1}$. Similarly,
\begin{displaymath}
(-\bar{U}_0 v_{x} + \varphi_{y})|_{y=0}=(-\bar{U}_0 u_{x} + \varphi_{x})|_{y=L_2},
\end{displaymath}
and since $v|_{y=0}=v|_{y=L_2}=0$, we deduce that $$\varphi|_{y=0}=\varphi|_{y=L_2}=0.$$
All these boundary conditions imply that $\varphi=0$ ($=$ constant). Finally, if ${\bf u} \in D(A^{0})$ then
\begin{displaymath}
\mathcal{A}^0{\bf u} =-\bar{U}_0( u_{x}, v_{x})\in H^0,
\end{displaymath}
and we conclude that ${\bf u}_x \in L^2(\mathcal{M'})^2$ and
\begin{equation}
\|{\bf u}_x\|_{L^2(\mathcal{M'})^2} = \|\mathcal{A}^0{\bf u}\|_{H^0}=\|{\bf u}\|_{D(A^{0})}.
\end{equation}
\end{rem}

\subsection{Some properties related to the operator $A^0$}
We show in this paragraph some properties related to the system $A^0{\bf u}={\bf F}=(F_1, F_2)$ that we will use in the proof of the positivity of the operators $A^0$ and $A^{0*}$. To do that, we aim to prove some lemmas analogous to Lemmas 2.1 and 2.2 in \cite{CST09}.
\begin{lem}\label{lem-maps}
The operator $A^0$ maps $D(A^0)$ onto $H^0$.
\end{lem}
\begin{proof}
Let ${\bf F}=(F_1, F_2) \in H^0$. Then, we look for  ${\bf u}=(u,v) \in D(A^0)$ such that $A^0{\bf u}={\bf F}$. Hence, we obtain the following system for ${\bf u}$:
\begin{align}\label{Au=F}
\left\{
\begin{array}{l}
-\bar{U}_0 u_{x} -fv+ \varphi_{x}=F_1,\\
-\bar{U}_0 v_{x} +fu+ \varphi_{y}=F_2,\\
u_x+v_y=0.
\end{array}%
\right.
\end{align}
First, we assume that ${\bf F}$ is smooth enough. We then take the curl of the first two equations of (\ref{Au=F}) and use (\ref{Au=F})$_3$; we obtain
\begin{equation}\label{elliptic}
\left \{
    \begin{array}{l}
    -\bar{U}_0 \Delta v=F_{2x} - F_{1y}, \ \text{in} \ (0,L_1)\times(0,L_2)\\
    v=0, \ \text{at} \ y=0,L_2,\\
    v|_{x=0}=v|_{x=L_1}.
    \end{array}
\right.
\end{equation}
Hence, the solution $v$ exists and is unique, and the solution $u$ is deduced using the incompressibility condition (\ref{Au=F})$_3$ and the boundary conditions (\ref{mode_zero_limit_bdry}), (\ref{mode_zero_limit_bdry2}). The pressure $\varphi$ can be also easily obtained from (\ref{Au=F}) and (\ref{elliptic}).\\

Now, for the general case ${\bf F} \in H^0$, consider $({\bf F}_m)_m \subset \mathcal{V}$ a sequence of functions such that ${\bf F}_m \to {\bf F}$ in $H^0$, as $m \to \infty$. For each ${\bf F}_m$, there exists ${\bf u}_m=(u_m,v_m) \in D(A^0)$ such that $A^0{\bf u}_m={\bf F}_m$. The process explained before for smooth functions, that is regularization and convolution (see e.g. the proof of Lemma \ref{pos-A}), gives the existence of $u_m, v_m$ and $\varphi_m$. We infer from (\ref{elliptic}) written for $v_m, {\bf F}_m$ that the $v_m$ are bounded in $H^1(\mathcal{M'})$, and thus the $u_m$ are bounded in $L^2(\mathcal{M'})$. It is then easy to see that $u_m, v_m, \varphi_m$ converge to $u, v, \varphi$, respectively, such that $A^0{\bf u}=A^0(u, v)={\bf F}$.
\end{proof}

\subsection{The positivity of $A^0$ and $A^{0*}$}
In this subsection, we aim to prove the positivity of the operators $A^0$ and $A^{0*}$. With these properties the Hille-Phillips-Yosida Theorem can be applied showing thus the existence of solution of the zeroth mode.
\begin{lem}\label{pos-A0}
The operators $A^0$ and $A^{0*}$ are positive.
\end{lem}
\begin{proof}
First, we note that the proof of positivity for $A^{0*}$ is similar to the proof for $A^0$. Thus, we only prove this property for the operator $A^0$.
For that purpose, we first observe that for ${\bf u} \in D(A^{0})$ sufficiently smooth, we can easily see that
\begin{eqnarray}\label{posit-regular0}
\ \ (A^{0}{\bf u},{\bf u})\!\!&\!=\!&\! \frac{\bar{U}}{2} \int_0^{L_2}\![u^2+v^2]_{x=0}^{x=L_1} \ud y +\! \int_0^{L_2}\![\varphi u]_{x=0}^{x=L_1} \ud y+\! \int_0^{L_1}\![\varphi v]_{y=0}^{y=L_2} \ud x\\&\!=\!&(\text{Thanks to (\ref{mode_zero_limit_bdry}), (\ref{mode_zero_limit_bdry2}) and (\ref{mode_zero_limit_bdry1})})\nonumber\\
&\!=\!&  0.\nonumber
\end{eqnarray}
Now, for ${\bf u} \in D(A^{0})$, thanks to Lemma \ref{lem-maps}, we set ${\bf F} = A^0{\bf u} \in H^0$. Using the density of $\mathcal{V}$ in $H^0$, we can approximate ${\bf F}$ by a sequence of functions $({\bf F}_m)_m \subset \mathcal{V}$ such that ${\bf F}_m \to F$ in $H^0$, as $m \to \infty$. For each ${\bf F}_m$, using again Lemma \ref{lem-maps}, we have the existence of ${\bf u}_m=(u_m,v_m) \in D(A^0) \cap \mathcal{V}$ such that $A^0{\bf u}_m={\bf F}_m$. Thanks to (\ref{posit-regular0}), we deduce that $(A^0{\bf u}_m, {\bf u}_m) \ge 0$.\\
Finally, it is easy to see that ${\bf u}_m \to {\bf u}$ in $D(A^{0})$ and thus deduce the positivity of $A^{0}$.
\end{proof}

\subsection{Existence and regularity of solutions for the zeroth mode}

It is now straightforward to conclude the existence of the  solution of the zero mode equations by simply applying the Hille-Phillips-Yosida Theorem for which all the hypotheses are already verified. We remember here to introduce the Coriolis force which corresponds to a bounded linear operator in $H^0$.
\begin{prop}\label{exist-mode0}
For given ${\bf F}_{\bf u_0}=(F_{u_0}, F_{v_0})^T$ such that ${\bf F}_{\bf u_0} \in \mathcal{C}^1(0,T;{H^0})$ and ${\bf \tilde{u}_0}= (\tilde{u}_0, \tilde{v}_0) \in D(A^0)$, there exists a unique solution ${\bf u_0}=(u_0, v_0)$ to the system (\ref{mode_zero}) together with boundary and initial conditions (\ref{mode_zero_limit_bdry})-(\ref{mode_zero_limit_bdry1}) and (\ref{ini_cond_mode}) (for $n=0$) such that
\begin{align}\label{exist-sol-0}
\left\{%
\begin{array}{c}
{\bf u_0} \in \mathcal{C}^1([0,T];{H^0}) \cap \mathcal{C}^0([0,T];D(A^0)),\\
\displaystyle{\frac{\ud {\bf u_0}}{\ud t}\in \mathcal{C}^0([0,T];H^0), \ \forall \ T>0}.
\end{array}%
\right.
\end{align}

Furthermore, by Remark \ref{rem-reg-mode0}, ${\bf u_0}_x \in \mathcal{C}^0([0,T];{H^0})$ and the sum of the norms of ${\bf u_0}$ in all the spaces in (\ref{exist-sol-0}) is bounded, up to a multiplicative constant, by the sum of the norms of the data, ${\bf F}_{\bf u_0}, {\bf F}_{\bf u_0}^{'},{\bf \tilde{u}_0}$.
\end{prop}

In what follows, we will derive some regularity results for the zero mode solution which are necessary to prove the regularity of the global solution. First, we rewrite here the system (\ref{mode_zero}) satisfied by the zero mode solution. For simplicity we omit in the following of this subsection the subscript $0$ assigned to the zero mode solution ${\bf u_0}=(u_0, v_0)$:
\begin{align}\label{mode_zero-reg}
\left\{%
\begin{array}{l}
u_{t} + \bar{U}_0u_{x} - f v + \phi_{x}  = F_{u},\\
v_{t} + \bar{U}_0 v_{x} + f u + \phi_{y}  =F_{v},\\
u_{x}+v_{y} = 0,
\end{array}%
\right.
\end{align}
and the boundary and initial conditions
\begin{align}\label{mode_zero-reg-bc}
\left\{%
\begin{array}{l}
u(0,y,t)=u(L_1,y,t), \ \forall \ y \in (0,L_2), \ \forall \ t \in (0,T),\\
v(0,y,t)=v(L_1,y,t),\ \forall \ y \in (0,L_2),  \ \forall \ t \in (0,T),\\
v = 0\ \text{at $y=0,L_2$,}\ \forall \ x \in (0,L_1),  \ \forall \ t \in (0,T),\\
(u,v) =
(\tilde{u},\tilde{v})(x,y)\ \text{at
$t=0$}.
\end{array}%
\right.
\end{align}
\begin{prop}\label{prop-reg-mode0}
We are given ${\bf F}_{\bf u}=(F_{u}, F_{v})^T$ and ${\bf u}$ such that
\begin{displaymath}
\left\{\begin{array}{l}
{\bf F}, {\bf F}_t, {\bf F}_{xx}, {\bf F}_{xxt} \in \mathcal{C}^0(0,T;{H^0}),\\
{\bf \tilde{u}} \in H^2(\mathcal{M'}) \cap D(A^0), \ {\bf \tilde{u}}_{x}, {\bf \tilde{u}}_{xx} \in D(A^0).
\end{array}
\right.
\end{displaymath}
Then the solution of (\ref{mode_zero-reg})-(\ref{mode_zero-reg-bc}) satisfies the following regularity properties
\begin{equation}\label{u0-reg}
{\bf u}_{xt} \in \mathcal{C}^0([0,T]; L^2(\mathcal{M'})^2), \ {\bf u} \in \mathcal{C}^0([0,T]; H^2(\mathcal{M'})^2) \ \text{and} \ \phi \in \mathcal{C}^0([0,T]; H^2(\mathcal{M'})).
\end{equation}
\end{prop}
\begin{proof}
We begin by observing that ${\bf u}_{x^k}, {\bf u}_{t^l}, {\bf u}_{x^kt^l}$ satisfy the same equations and boundary conditions (\ref{mode_zero-reg})-(\ref{mode_zero-reg-bc}) as ${\bf u}$ with ${\bf \tilde{u}}$ and ${\bf F}={\bf F}_{\bf u}$ replaced respectively by ${\bf \tilde{u}}_{x^k}, {\bf \tilde{u}}_{t^l}, {\bf \tilde{u}}_{x^kt^l}$ and ${\bf F}_{x^k}, {\bf F}_{t^l}, {\bf F}_{x^kt^l}$. Hence the conclusions of the existence result (\ref{exist-sol-0}) hold also for any derivative of ${\bf u}$ w.r.t $x$ or $t$ provided we make the suitable assumptions on the data. We will call this property {\it the invariance} property of our system (\ref{mode_zero-reg})-(\ref{mode_zero-reg-bc}) under differentiation w.r.t $x$ or $t$.\\

Now if we call $(\mathscr{H}_0)$ the set of hypotheses made in Proposition \ref{exist-mode0}, then we already have
\begin{equation}\label{4.14a}
{\bf u}, {\bf u}_{x}, {\bf u}_{t} \in \mathcal{C}^0([0,T]; L^2(\mathcal{M'})^2).
\end{equation}
With an obvious notation, the hypotheses $(\mathscr{H}_0)_x$ imply in addition that
\begin{equation}
{\bf u}_{xx}, {\bf u}_{xt} \in \mathcal{C}^0([0,T]; L^2(\mathcal{M'})^2),
\end{equation}
so that the first of the conditions (\ref{u0-reg}) is fulfilled.

Before proving that ${\bf u} \in \mathcal{C}^0([0,T]; H^2(\mathcal{M'})^2)$, let us first prove that \linebreak ${\bf u} \in \mathcal{C}^0([0,T]; H^1(\mathcal{M'})^2)$. In view of (\ref{4.14a}) we only need to prove that ${\bf u}_{y} \in \mathcal{C}^0([0,T]; L^2(\mathcal{M'})^2)$. Since $(u,v,\phi)$ satisfy (\ref{mode_zero-reg}), we obtain from (\ref{mode_zero-reg})$_2$ that $\phi_{y} \in \mathcal{C}^0([0,T]; L^2(\mathcal{M'})^2)$.

Now, using (\ref{mode_zero-reg})$_{1,2}$, we deduce from the hypotheses $(\mathscr{H}_0)$ that $\phi_{x}, \phi_{y} \in$ \linebreak $ \mathcal{C}^0([0,T]; L^2(\mathcal{M'}))$. Thanks to (\ref{mode_zero-reg})$_{3}$, we infer that $v_{y}=-u_x \in \mathcal{C}^1([0,T]; L^2(\mathcal{M'}))$. Thus, it only remains to prove that $u_{y} \in \mathcal{C}^0([0,T]; L^2(\mathcal{M'}))$. For that, we rewrite (\ref{mode_zero-reg})$_{1}$ differentiated in $y$ as a transport equation for $\theta =u_y$ which reads as follows:
\begin{equation}\label{transport-eq}
\left\{
\begin{array}{l}
\theta_{t} + \bar{U}_0 \theta_{x} = f v_y - \phi_{xy}  + F_{uy}, \ \text{in} \ (0,T)\times(0,L_1),\\
\theta \ \text{is periodic in} \ x,\\
\theta|_{t=0}=\tilde{{\bf u}}_y.
\end{array}
\right.
\end{equation}
The solution of (\ref{transport-eq}) is classically obtained by integrating along the characteristics of this hyperbolic equation. Hence, the regularity of $\theta$ is exactly the same as that of the RHS in (\ref{transport-eq})$_{1}$ and for the initial data (\ref{transport-eq})$_{3}$ combined together. We obtain that $\phi_{xy} \in \mathcal{C}^0([0,T]; L^2(\mathcal{M'}))$ by using the hypotheses $(\mathscr{H}_0)_x$. Therefore, we conclude that $u_{y} \in \mathcal{C}^0([0,T]; L^2(\mathcal{M'}))$ by assuming $(\mathscr{H}_0), (\mathscr{H}_0)_x$ and that ${\bf \tilde{u}}_y \in L^2(\mathcal{M'})$.

Finally the whole solution of (\ref{mode_zero-reg}) satisfies ${\bf u} \in \mathcal{C}^0([0,T]; H^1(\mathcal{M'})^2)$ and $\phi \in \mathcal{C}^0([0,T]; H^1(\mathcal{M'}))$.

Now in order to prove that ${\bf u} \in \mathcal{C}^0([0,T]; H^2(\mathcal{M'})^2)$, we must prove that
\begin{equation}
{\bf u}_{xx}, {\bf u}_{xy}, {\bf u}_{yy} \in \mathcal{C}^0([0,T]; L^2(\mathcal{M'})^2).
\end{equation}
The property is already known for ${\bf u}_{xx}$ and we obtain it as before for ${\bf u}_{xy}$ by assuming $(\mathscr{H}_0)_{xx}$ and that ${\bf \tilde{u}}_{xy} \in L^2(\mathcal{M'})^2$. At this stage by differentiating (\ref{mode_zero-reg}) in $y$, we find that $\phi_{xy}, \phi_{yy} \in \mathcal{C}^0([0,T]; H^1(\mathcal{M'}))$, and $\phi_{xx} \in \mathcal{C}^0([0,T]; H^1(\mathcal{M'}))$ just by $(\mathscr{H}_0)_{x}$.
Hence $\phi \in \mathcal{C}^0([0,T]; H^2(\mathcal{M'}))$. The incompressibility equation differentiated in $y$ gives $v_{yy} \in \mathcal{C}^0([0,T]; L^2(\mathcal{M'}))$. For $u_{yy}$ we proceed as for $u_{y}$, differentiating  (\ref{mode_zero-reg})$_1$ twice in $y$. The result follows by assuming in addition that ${\bf \tilde{u}}_{yy} \in L^2(\mathcal{M'})$.

The Proposition \ref{prop-reg-mode0} is proved.
\end{proof}
\begin{rem}
Note that the proof of higher regularity in $x$ and $t$ can be done recursively thanks to the invariance property and the fact that each time the initial data are deduced from (\ref{mode_zero-reg})$_{1,2}$. However, the system (\ref{mode_zero-reg})-(\ref{mode_zero-reg-bc}) is not invariant under differentiation w.r.t $y$, but, we can continue to show the regularity w.r.t. $y$ by differentiating the equation in $y$ enough times and adding some hypotheses for the derivatives of the initial data as for  (\ref{transport-eq}).
\end{rem}

\section{Existence result for the (full) system (\ref{model0})}\label{sec4}

We investigate in this section the existence of solutions of system (\ref{model0}) associated with the initial conditions (\ref{ini_cond_mode}), the $x-$periodicity conditions (\ref{per_bdry_x}) and the $y-$ boundary conditions (\ref{mode_zero_limit_bdry})-(\ref{mode_zero_limit_bdry1}) for the zero mode, and (\ref{zeta})-(\ref{xhi}) for the modes $n\ge1$. To this end, we introduce, as in \cite{rtt08}, the decomposition of the solution $U$ (and its derivatives) in the form  $U=(U^0,U^I)$. Here $U^0$ and $U^I$ stand, respectively, for the zero mode solution of (\ref{mode_zero}) and the whole supercritical solution formed by summing all the $n$th modes, solutions of (\ref{moden}), for all $n \ge 1$. Note that $U^0=(u_0,v_0,0)$ and $U^I=(u^I,v^I,\psi^I)$.

\begin{rem}
Although the computations in this section are very similar to those developed in  \cite{rtt08} and \cite{CST09}, we choose to put them in an abstract form for the reader's convenience; more rigourous details can be found in Subsection 4.1 of \cite{rtt08}; see also \cite{CST09}.
\end{rem}

Now, the existence of solutions for the supercritical modes is proved in Section \ref{sec3} and the zero mode in Section \ref{sec-mode0}. It remains to prove a similar existence result for the concatenated solution $U=(u,v,\psi)$ as defined in (\ref{modal_form}). Hence, we introduce the operator $\mathbb{A}$ and its domain$D(\mathbb{A})$ in $\mathbb{H}$ which reads as below:
\begin{equation}\label{space-glob-1}
\mathbb{H} = H^0 \times H^I, \quad \text{and} \quad D(\mathbb{A}) = D(A^0) \times D(A^I),
\end{equation}
where $H^0$ and $D(A^0)$ are the function spaces related to the mode zero equations, they are given by (\ref{space-glob-2}) and (\ref{domain-A0}), respectively.
However, the function space $H^I$ is simply the orthogonal, in $L^2(\mathcal{M})$, of the space of functions independent of $z$ and periodic in $x$ (as in (\ref{per_bdry_x})), and $D(A^I)$ is the domain of the operator associated with the supercritical modes $n\ge1$,
\begin{equation}\label{space-glob-4}
\begin{array}{c}
D(A^I) = \{U^I=\{U_n\}_{n\ge1}, U_n \in L^2(\mathcal{M'})^3, \mathcal{A}_n U_n \in L^2(\mathcal{M'})^3, n\ge1,\\ U^I \ \text{satisfies} \ (\ref{zeta}) \ \text{and} \ (\ref{xhi}) \ \text{componentwise} \},
\end{array}
\end{equation}
where $\mathcal{A}_n$ is given by (\ref{operator-A}).

We endow the space $\mathbb{H}$ with its natural scalar product and norm defined as follows:
\begin{eqnarray}
(U,U^\sharp)_{\mathbb{H}}&=&((u,v,\psi),(u^\sharp,v^\sharp,\psi^\sharp))_{L^2(\mathcal{M})^3} \nonumber\\
&=& (u,u^\sharp)_{L^2(\mathcal{M})}+(v,v^\sharp)_{L^2(\mathcal{M})}
+\frac{1}{N^2}(\psi,\psi^\sharp)_{L^2(\mathcal{M})},\nonumber\\
|U|_{\mathbb{H}}&=&\left[ (U,U)_{\mathbb{H}} \right]^{1/2}.\nonumber
\end{eqnarray}
In a componentwise formulation as introduced above in the head of this paragraph, the norm in $\mathbb{H}$ can be also expressed as follows:
\begin{displaymath}
|U|_{\mathbb{H}}^2=|{\bf u_0}|^2_{L^2(\mathcal{M'})^2}+\sum^{\infty}_{n=1}|U_n|^2_{{\bf L^2(\mathcal{M'})}},
\end{displaymath}
where we recall again here that $U^0={\bf u_0}=(u_0,v_0)$ and we have $\psi_0 \equiv 0$.\\
The operator $\mathbb{A}$ is acting from its domain $D(\mathbb{A})$ into $\mathbb{H}$  componentwise as $\mathbb{A} U = (A^0 U^0, A^I U^I)$.

We now define the adjoint operator $\mathbb{A}^*$ of $\mathbb{A}$ and prove that $\mathbb{A}^*$ and $\mathbb{A}$ are positive using the previous positivity results already obtained componentwise.

It is easy to see that
\begin{displaymath}
D(\mathbb{A}^*) = D(A^{0*}) \times D(A^{I*}),
\end{displaymath}
where
\begin{displaymath}
D(A^{0*}) = \{{\bf \tilde{u}_0}=(\tilde{u}_0,\tilde{v}_0) \in H^0, \exists \ \Phi \in \mathcal{D}'(\mathcal{M'}) \ \text{s.t.} \ \mathcal{A}_0{\bf \tilde{u}_0}\in H^0 \ \text{and} \ \tilde{v}_0=0 \ \text{at} \ x=L_1 \},
\end{displaymath}
and
\begin{displaymath}
\begin{array}{c}
D(A^{I*}) = \{U^I=\{U_n\}_{n\ge1}, U_n \in L^2(\mathcal{M'})^3, \mathcal{A}_n U_n \in L^2(\mathcal{M'})^3, n\ge1,\\ U^I \ \text{satisfies} \ (\ref{zeta}), (\ref{xhi}) \ \text{and} \ (\ref{per_bdry_x})\}.
\end{array}
\end{displaymath}

We have the following:
\begin{thm}\label{th4.1}
The operator $-\mathbb{A}$ is the infinitesimal generator of a semigroup of contractions in $\mathbb{H}$.
\end{thm}
\begin{proof}
According to the well-known characterization of a semigroup of contractions (see e.g. \cite{yosida}), it suffices to prove the density of the spaces $D(\mathbb{A})$ and $D(\mathbb{A}^*)$ in $\mathbb{H}$ and the closure and the positivity of the operator $\mathbb{A}$ and its adjoint $\mathbb{A}^*$, namely
\begin{eqnarray}
&(\mathbb{A} U, U)_{\mathbb{H}} \ge 0, \quad \forall \, U \in D(\mathbb{A}),&\\
&(\mathbb{A}^* U, U)_{\mathbb{H}} \ge 0, \quad \forall \, U \in D(\mathbb{A}^*).&
\end{eqnarray}
First, the density of $D(\mathbb{A})$ and $D(\mathbb{A}^*)$ in $\mathbb{H}$ is straightforward and implies in addition the closure of the operators $\mathbb{A}$ and $\mathbb{A}^*$. Indeed, we proceed componentwise: for $D(A^0)$, it is already known that the space of $\mathcal{C}^{\infty}$ functions ${\bf u^0}=(u^0,v^0)$ which are periodic in $x$ and with compact support in $(0,L_2)$ and such that $u^0_x+v^0_y=0$ is dense in $H^0$. For $D(A^I)$, similarly, the space of $\mathcal{C}^{\infty}$ functions with compact support in $\mathcal{M'}$ is dense in $L^2(\mathcal{M'})^3$. Moreover, the positivity is guaranteed thanks to the positivity of the operators corresponding to each mode; the mode zero in Lemma \ref{pos-A0}  and the modes $n \ge 1$ in Lemmas \ref{pos-A} and \ref{pos-A*}.
\end{proof}
By setting as before $\mathbb{U}=(u, v,\psi), \mathbb{F}=(F_u, F_v,F_\psi)$ and $\tilde{\mathbb{U}}=(\tilde{u}, \tilde{v}, \tilde{\psi})$, we may rewrite the system
(\ref{model0}) in an abstract form as follows:
\begin{equation}\label{evolution}
\left\{\begin{array}{l}\displaystyle{\frac{d \mathbb{U}}{d t} + \mathbb{A} \mathbb{U}= \mathbb{F},}\\
\mathbb{U}(0)=\tilde{\mathbb{U}},
\end{array}\right.
\end{equation}
where $\mathbb{A}$ is an unbounded linear operator from  $D(\mathbb{A}) \subset \mathbb{H}$ into the Hilbert space $\mathbb{H}$ defined above (see (\ref{space-glob-1})---(\ref{space-glob-4})).

A straightforward consequence of Theorem \ref{th4.1} is the following corollary which summarizes the existence result for the system (\ref{model0}).

\begin{cor}\label{corol}
For given $\mathbb{F}=(F_u, F_v, F_{\psi})^T$ such that $\mathbb{F} \in \mathcal{C}^1([0,T];{\mathbb{H}})$ and $ \tilde{\mathbb{U}}= \{\tilde{u}_n\}_{n\in \en} \in D(\mathbb{A})$, there exists a unique solution $\mathbb{U}=(u, v, {\psi})$ to the system (\ref{model0}) (or equivalently (\ref{evolution})) with
\begin{align}\label{exist-sol-glob}
\left\{%
\begin{array}{c}
\mathbb{U} \in \mathcal{C}^1([0,T];{\mathbb{H}}) \cap \mathcal{C}^0([0,T];D(\mathbb{A}),\\
\displaystyle{\frac{\ud \mathbb{U}}{\ud t}\in \mathcal{C}^0([0,T];\mathbb{H}), \ \forall \ T>0}.
\end{array}%
\right.
\end{align}
\end{cor}

We recall here that the estimates (\ref{bound1})-(\ref{bound2}), obtained for each mode solution, guarantee that similar estimates to  (\ref{bound1})-(\ref{bound2}) hold for the solution $\mathbb{U}$ of (\ref{evolution}) provided that $\tilde{\mathbb{U}} \in D(\mathbb{A}) \cap H^1(\mathcal{M})^3$ and  $\mathbb{F} \in \mathcal{C}^1([0,T]; L^2(\mathcal{M})^3)$.

\section{Further regularity for the limit problem
(\ref{model0})}\label{sec6} We establish in this section some further
 regularity results related to the solution $(u, v,w,
\phi,\psi)$  of the limit system (\ref{model0}) in view of reaching the regularity hypothesis assumed and used in \cite{HJT}, see (\ref{a-regular}) below. As mentioned before,  it is sufficient to look for the regularity of the components  $u, v$ and $\psi$ as the other components $\phi$ and $\psi$ can be simply deduced using (\ref{model0})$_{4,5}$ as we did in Remark \ref{rem21}. Furthermore, we notice that the space $\mathbb{H}$ is a subspace of $L^2(\mathcal{M})^3$ which is isomorphic to $L^2(\mathcal{M})^3$. Indeed, each element of $L^2(\mathcal{M})^3$ can be written using the modal decomposition (\ref{modal_form})-(\ref{modal_form_b}) as an element of $\mathbb{H}$, see \cite{rtt08} for more details.

We start by recalling the system (\ref{model0}) and especially we write it in a nonlocal form regarding some terms in the equations and we recall also the boundary conditions that we already introduced in Section \ref{sec2}. More precisely, we have
\begin{align}\label{model0-reg}
\left\{%
\begin{array}{l}
u_t + \bar{U}_0u_x - f v -\int_z^0 \psi_x (x,z') \ud z'] = F_u,\\
v_t + \bar{U}_0 v_x + f u -\int_z^0 \psi_y (x,z^{'}) \ud z']   =
F_v,\\
\psi_t + \bar{U}_0\psi_x + N^2 \int_z^0 [u_x (x,z')+v_y (x,z')] \ud z'  = F_{\psi},
\end{array}%
\right.
\end{align}
for which we associated the boundary conditions (\ref{bc}) given mode by mode, that we rewrite in the nonlocal form, $\forall \ t \in (0,T), \ \forall \ x \in (0,L_1),$
\begin{equation}\label{bc-global}
\left\{\hspace{-.2cm}
\begin{array}{l}
\displaystyle{\int_{-L_3}^0 v(t;x,0,z)\mathcal{U}_n (z) \ud z -
\frac{1}{N}\int_{-L_3}^0 \psi(t;x,0,z)\mathcal{W}_n
(z) \ud z=0,}\vspace{.2cm}\\
\displaystyle{\int_{-L_3}^0 v(t;x,L_2,z)\mathcal{U}_n (z) \ud z +
\frac{1}{N}\int_{-L_3}^0 \psi(t;x,L_2,z)\mathcal{W}_n (z) \ud z=0,}
\end{array}\right.
\end{equation}
together with the $z-$ boundary conditions (\ref{z-bc}), the periodicity condition in $x$ (see (\ref{per_bdry_x})) and the initial data written mode by modes, that is (\ref{ini_cond_mode}). Note that the boundary conditions (\ref{bc-global}) are valid for all the modes $n \in \en$  including the zero mode since $\psi_0 \equiv 0$.\\




For every $\mathbb{U} \in D(\mathbb{A})$, $\mathbb{A} \mathbb{U} =\mathbb{\mathcal{A}} \mathbb{U}=(A^0 U^0, A^I U^I) \in \mathbb{H}$ is defined as follows:
\begin{equation}\label{operator}
A^0 U^0 =\mathbb{P}_{H^0}\left[\bar{U}_0 U^0_x + f (U^0)^{\perp}
+ \nabla \Phi^0\right],
\end{equation}
and
\begin{equation}\label{operator}
A^I U^I =\left(\begin{array}{l}\bar{U}_0u^I_x - f v^I
-\int_z^0 \psi^I_x (x,z') \ud z'\\
\bar{U}_0 v^I_x + f u^I-\int_z^0 \psi^I_y (x,z') \ud z'\\
 \bar{U}_0\psi^I_x + N^2 \int_z^0 [u^I_x (x,z')+v^I_y (x,z') \ud z'
\end{array}\right),
\end{equation}
where $\mathbb{P}_{H^0}$ is the orthogonal projector from $(L^2(\mathcal{M}))^2$ onto $H^0$.\\
Now, by Corollary \ref{corol}, we have the
existence and uniqueness of a solution $\mathbb{U}$ of (\ref{evolution}) satisfying (\ref{exist-sol-glob}) for which we will prove in the subsequent some additional regularity results.

We arrive now at our final (and main) aim in this article which is to prove the following regularity properties (\ref{a-regular}) of the solution $\mathbb{U}$ of equations (\ref{model0-reg}), (\ref{bc-global}), an assumption which was made in \cite{HJT}:
\begin{align}\label{a-regular}
\left\{
\begin{array}{l}
\mathbb{U}_{xt} \in L^{\infty}(0,T; \mathbb{H}),\\
\mathbb{U} \in L^{\infty}(0,T; H^2(\mathcal{M})),\\
\mathbb{U}_{tz}, \mathbb{U}_{xxz}, \mathbb{U}_{txz} \in L^{\infty}(0,T; \mathbb{H}),\\
\mathbb{U}_{xzz} \in L^{\infty}(0,T; \mathbb{H}).\\
\end{array}\right.
\end{align}

The proof of (\ref{a-regular}) will necessitate numerous regularity hypotheses on the data  $\mathbb{F}$ and $\tilde{\mathbb{U}}=\mathbb{U}(0)$. To state them as simply as possible, we will call $(\mathscr{H})$ the set of hypotheses of Corollary \ref{corol}, namely
\begin{displaymath}
(\mathscr{H}) \quad \left\{%
\begin{array}{l}
\mathbb{F} \in \mathcal{C}^1([0,T];{\mathbb{H}}), \\
\tilde{\mathbb{U}} \in D(\mathbb{A}) \cap H^1(\mathcal{M})^3.
\end{array}%
\right.
\end{displaymath}
With an obvious notation, we will call $\mathscr{H}_x$ or $\mathscr{H}_z$, etc., the same hypotheses made on $\mathbb{F}_x, \tilde{\mathbb{U}}_x$ or $\mathbb{F}_z, \tilde{\mathbb{U}}_z$. A little more delicate is the hypothesis $\mathscr{H}_t$ for $\mathbb{U}_t$. Namely $\mathscr{H}_t$ means that $\mathbb{F}_t \in \mathcal{C}^1([0,T];{\mathbb{H}})$ so that $\mathbb{F} \in \mathcal{C}^2([0,T];{\mathbb{H}})$ and $\mathbb{U}_t(t=0) \in D(\mathbb{A}) \cap H^1(\mathcal{M})^3$. But $\mathbb{U}_t(t=0)$ is given by equation (\ref{evolution}) written at time $t=0$, and the condition on $\mathbb{U}_{t}(0)$ reads:
\begin{align}
\mathbb{U}_{t}(0) = - \mathbb{A} \tilde{\mathbb{U}} + \mathbb{F}(0) \in D(\mathbb{A}) \cap H^1(\mathcal{M})^3.
\end{align}

We now state the main regularity result.
\begin{thm}\label{a-prop}
With the notation above, we assume
\begin{align}
\left\{(\mathscr{H}), (\mathscr{H}_{x}), (\mathscr{H}_{xx}), (\mathscr{H}_{xxx}), (\mathscr{H}_{tx}), (\mathscr{H}_{txx}), (\mathscr{H}_{zz}), (\mathscr{H}_{xzz}), (\mathscr{H}_{txzz}), (\mathscr{H}_{xxzz})\right\}.
\end{align}
Then, the solution $\mathbb{U}$ of  (\ref{model0-reg})-(\ref{bc-global}) satisfies (\ref{a-regular}).
\end{thm}

\begin{proof}{}
First, we note that the equations and boundary conditions
satisfied by $\mathbb{U}$, namely (\ref{model0-reg}) and (\ref{bc-global}), are invariant by differentiation in time and in $x$. They are also invariant by {\it double} differentiation in $z$
except, as we will see below, for the appearance of unimportant lower order terms. Moreover, the initial conditions (\ref{evolution})$_2$ (or (\ref{ini_cond_mode})) are invariant with respect to the differentiation in $x$ and $z$ and provided of course that we have the necessary regularity while differentiating with respect to theses variables.

To show the regularity of the global solution $\mathbb{U}$, solution of (\ref{model0-reg}), we will need to prove the componentwise regularity for each mode for the quantities appearing in (\ref{a-regular}). However, we begin by proving the properties that are the same for all the modes. For that, we will first focus on (\ref{a-regular})$_{1}$ which will be useful for (\ref{a-regular})$_{3}$. For (\ref{a-regular})$_{1}$, it suffices to observe that by differentiation in $x$ of
(\ref{evolution}), $\mathbb{U}_{x}$ satisfies the same equation as $\mathbb{U}$,
and therefore $(\mathscr{H}_x)$ yields for $\mathbb{U}_x$ the conclusions similar to (\ref{exist-sol-glob}) for
$\mathbb{U}$, that is
\begin{equation}\label{conc-reg}
\mathbb{U}_{x} \in \mathcal{C}^1([0,T]; \mathbb{H}) \cap \mathcal{C}^0([0,T]; D(\mathbb{A})), \quad
\textnormal{and} \quad \mathbb{U}_{xt} \in \mathcal{C}^0([0,T]; \mathbb{H}).
\end{equation}

In view of proving (\ref{a-regular})$_{2}$ we first observe that, in a similar way $(\mathscr{H}_{xx})$ implies that $\mathbb{U}_{xx} \in \mathcal{C}^0([0,T];L^2(\mathcal{M})^3)$. This gives the first term in the $H^2-$ regularity assertion (\ref{a-regular})$_{2}$. Then, we prove the regularity of $\mathbb{U}$ w.r.t. the $y$ variable. For that, we go back to the modal decomposition of the solution of (\ref{model0-reg}), mode by mode. First, we start by the mode zero for which we already proved Proposition \ref{prop-reg-mode0} which showed all the needed regularity for this mode.

%

Second, we can prove the regularity of the $n$th modes, solutions of (\ref{mode-n}) (we reintroduce here the dependence in $n$ for the solution of this system) for all $n \ge 1$, with respect to the $y$-differentiation. To this end, we infer from (\ref{moden}), (\ref{exist-sol}) and (\ref{conc-reg}) that $\psi_{ny}, v_{ny} \in \mathcal{C}^0([0,T]; L^2(\mathcal{M'}))$ (see (\ref{operator-A}) and (\ref{domain-A})). Then, the same process as for the zero mode can be carried out by simply differentiating (\ref{mode-n})$_1$ in $y$. This yields $u_{ny} \in \mathcal{C}^0([0,T]; L^2(\mathcal{M'}))$ resulting from $(\mathscr{H}_{xx})$. We also obtain by $(\mathscr{H}_{xx})$ that $\psi_{nxy}, v_{nxy}$ belong to $\mathcal{C}^0([0,T]; L^2(\mathcal{M'}))$. Similarly, $\psi_{nyy}, v_{nyy} \in \mathcal{C}^0([0,T]; L^2(\mathcal{M'}))$ thanks to (\ref{mode-n})$_{2,3}$ differentiated in $y$ and the previous observations (using $(\mathscr{H}), (\mathscr{H}_{x}), (\mathscr{H}_{xx})$). Finally, the regularity for $u_n$, namely $u_{nyy} \in \mathcal{C}^0([0,T]; L^2(\mathcal{M'}))$, is deduced from (\ref{mode-n})$_1$ as previously ({\it i.e.} writing the equation (\ref{mode-n})$_1$ differentiated twice in $y$ as a transport equation).

Before we continue with (\ref{a-regular})$_2$, we prove the regularity in $z$ and start by proving that
\begin{equation}\label{reg-U-lap}
\mathbb{U}_{zz} \in \mathcal{C}^0([0,T];  \mathbb{H}).
\end{equation}
To prove the condition in (\ref{reg-U-lap}), we consider the equations and boundary conditions satisfied by
$\mathbb{U}_{zz}$. By double differentiation in $z$, using the series expansions (\ref{modal_form}), and assuming enough regularity in $z$, we see that
$\mathbb{U}_{zz}$ satisfies the same boundary conditions as $\mathbb{U}$ at $z=0,-L_3$ and on the lateral boundary.
Furthermore  $\mathbb{U}_{zz}$ satisfies
\begin{equation}\label{mod-evolution}
\left\{\begin{array}{l}\displaystyle{\frac{d \mathbb{U}_{zz}}{d t} + \mathbb{A}^{\flat} \mathbb{U}_{zz}=\mathbb{F}_{zz},}\\
\mathbb{U}_{zz}(0)=\tilde{\mathbb{U}}_{zz},
\end{array}\right.
\end{equation}
where $\mathbb{A}^{\flat}=\mathbb{A}+(\mathbb{P}_{H_u}(\psi_{xz}(z=0)),\mathbb{P}_{H_v}(\psi_{yz}(z=0)),
\mathbb{P}_{H_{\psi}}(-N^2[u_{xz}(z=0)+
v_{yz}(z=0)]))^{t}I$, where $\mathbb{P}_{X}$ is the orthogonal projector from $L^2(\mathcal{M})$ onto $X$.\\
Now since the difference between $\mathbb{A}^{\flat}$ and $\mathbb{A}$ is
independent of $z$, this difference is in the kernel of
$\mathbb{P}_H=(\mathbb{P}_{H_u},\mathbb{P}_{H_v},\mathbb{P}_{H_{\psi}})^t$. Hence applying the operator $\mathbb{P}_H$ to
each side of (\ref{mod-evolution})$_1$, we find that
\begin{equation}\label{mod-evolution-bis}
\left\{\begin{array}{l}\displaystyle{\frac{d \mathbb{U}_{zz}}{d t} + \mathbb{A} \mathbb{U}_{zz}= \mathbb{F}_{zz},}\\
\mathbb{U}_{zz}(0)=\tilde{\mathbb{U}}_{zz},
\end{array}\right.
\end{equation}
which is again the same equation as (\ref{evolution}). Then the hypotheses $(\mathscr{H}_{zz}), (\mathscr{H}_{zzx}), (\mathscr{H}_{zzxx})$, imply all the above properties for $\mathbb{U}_{zz}$, namely (\ref{reg-U-lap}) and also
\begin{equation}\label{reg-U-zz}
\mathbb{U}_{tzz}, \mathbb{U}_{zz}, \mathbb{U}_{txzz}, \mathbb{U}_{xzz}, \mathbb{U}_{xxzz} \in \mathcal{C}^0([0,T]; L^2(\mathcal{M})^3).
\end{equation}
Hence (\ref{a-regular})$_4$ is proven. Also by interpolation in $z$ with the previous results, we see that
\begin{equation}\label{reg-U-z}
\mathbb{U}_{tz}, \mathbb{U}_{z}, \mathbb{U}_{txz}, \mathbb{U}_{xz}, \mathbb{U}_{xxz} \in \mathcal{C}^0([0,T]; L^2(\mathcal{M})^3).
\end{equation}
This proves (\ref{a-regular})$_3$.

We now complete the proof of (\ref{a-regular})$_2$. We need to show that the following derivatives belong to $\mathcal{C}^0([0,T]; L^2(\mathcal{M})^3)$:
\begin{displaymath}
\mathbb{U}_{xz}, \mathbb{U}_{yz}, \mathbb{U}_{zz}.
\end{displaymath}
The regularity of $\mathbb{U}_{zz}$ is provided by (\ref{reg-U-zz}), that of $\mathbb{U}_{xz}$ by (\ref{reg-U-z}). To prove that $\mathbb{U}_{yz} \in \mathcal{C}^0([0,T]; L^2(\mathcal{M})^3)$ we proceed by interpolation between $\mathbb{U}_{y} \in \mathcal{C}^0([0,T]; L^2(\mathcal{M})^3)$ which was provided by $(\mathscr{H}_{x})$ and $(\mathscr{H}_{tx})$ and $\mathbb{U}_{yzz} \in \mathcal{C}^0([0,T]; L^2(\mathcal{M})^3)$ which is provided by $(\mathscr{H}_{xxzz})$ and $(\mathscr{H}_{txzz})$.

Theorem \ref{a-prop} is proven.
\end{proof}

\section*{Acknowledgements.}
This work was supported in part by NSF Grants DMS 1510249, and by the Research Fund of Indiana University. Chang-Yeol Jung was supported under the framework of international  cooperation program managed by the National Research Foundation of Korea (2015K2A1A2070543) and supported by the National Research Foundation of Korea grant funded by the Ministry of Education (2015R1D1A1A01059837).

\end{document}